\documentclass[11pt,reqno]{amsart}

\usepackage{amssymb,amsthm,amsmath}

\numberwithin{equation}{section}

\theoremstyle{plain}
\newtheorem{prop}{Proposition}[section]
\newtheorem{lem}[prop]{Lemma}

\newtheorem{thm}[prop]{Theorem}

\theoremstyle{definition}

\theoremstyle{remark}

\DeclareMathOperator{\SL}{SL}

\newcommand\A{\mathcal{A}}

\newcommand\N{\mathbb{N}}
\newcommand\Q{\mathbb{Q}}
\newcommand\R{\mathbb{R}}
\newcommand\Z{\mathbb{Z}}

\begin{document}
\title[]{Effective uniqueness of Parry measure and exceptional sets in ergodic theory.}

\author{Shirali\@ Kadyrov}

\address[SK]{Department of Mathematics,
Nazarbayev University,
Astana, Kazakhstan}

\email[SK]{shirali.kadyrov@nu.edu.kz}

\begin{abstract}
It is known that hyperbolic dynamical systems admit a unique invariant probability measure with maximal entropy. We prove an effective version of this statement and use it to estimate an upper bound for Hausdorff dimension of exceptional sets arising from dynamics.
\end{abstract}

\maketitle

\section{Introduction}\label{sec:intro}

Dynamical systems theory is a way to investigate time dependence of orbits in a given space acting under certain fixed rules. In a system with chaotic behaviour it is often difficult to understand the orbit of a given single point and instead it is useful to consider invariant measures. In general, there is an abundance of invariant measures and some of them are very special, such as the Parry measure (measure of maximal entropy). It is an interesting problem to study how the remaining measures relate to this measure. This will be our main goal in this paper.

We first define the symbolic space $(\Sigma_A,\sigma)$. For $s \ge 2$ let $A$ be an $s \times s$ square matrix with entries zeros and ones.  Set $\Lambda=\{1,2,\dots,s\}$ and define
$$\Sigma_A:=\left\{  x =(x_n)_{n=-\infty}^\infty \in \Lambda^\Z :  A(x_{n}, x_{n+1})=1, \forall n \in \Z\right\},$$
where $A(i,j)$ is the $(i,j)$th entry of $A$. Endowed with product topology $\Sigma_A$ is a compact space. The shift map $\sigma:\Sigma_A \to \Sigma_A$ is given by $\sigma((x_n)_{n=-\infty}^\infty)=(x_{n+1})_{n=-\infty}^\infty.$  The pair $(\Sigma_A,\sigma)$ is called a \emph{subshift of finite type}. We similarly define the one sided analog as follows: let $\Sigma_A^+=\{x \in \Lambda^\N: A(x_n,x_{n+1})=1, \forall n \in N\}$ and by abuse of notation we define the one sided shift $\sigma:\Sigma_A^+ \to \Sigma_A^+$ by $\sigma((x_n)_{n=0}^\infty)=(x_{n+1})_{n=0}^\infty.$ We now introduce a metric on these spaces. For a given $\theta>1$ define a metric $d_\theta$ on $\Sigma_A$ by 
$$d_\theta(x,y)=\theta^{-t(x,y)} \text{ where } t(x,y)=\max\{n \ge 0: x_i=y_i, |i|<n \}.$$
We also define a metric on $\Sigma_A^+$ by 
$$d_\theta(x,y)=\theta^{-t(x,y)} \text{ where } t(x,y)=\max\{n \ge 0: x_i=y_i, 0\le i<n \}.$$
The matrix $A$ is said to be \emph{irreducible} if  for each pair $(i,j)$ there exists $n\ge 1$ such that $A^n(i,j)>0$. We say that $A$ is \emph{aperiodic} if $A(i,i)=1$ for all $i=1,2,\dots,s$. Throughout the paper we assume that $A$ is both irreducible and aperiodic. By abuse of notation we let $m$ be the Parry measure on both $\Sigma_A$ and $\Sigma_A^+$. For a more precise definition of Parry measure see \S~\ref{sec:parry}. This is the unique measure of maximal entropy  \cite[Theorem~8.10]{Walter}. Our first goal in this paper is to prove the following effective version of the uniqueness statement. To this end we introduce a norm $|\cdot |_\theta$ as follows: For a continuous function $g$ on $\Sigma_A^+$ and any $n \ge 0$ we define ${\rm var}_n g=\sup \{|g(x)-g(y)|:x_i=y_i, 0\le i<n\}$. For any Lipschitz function $g$ on $\Sigma_A^+$ we define the seminorm $|g|_\theta$ by
$$|g|_\theta=\sup \left\{\frac{{\rm var}_n g}{\theta^n}: n \ge 0\right\}.$$ 
Note, that $|g|_\theta$ is the least Lipschitz constant. We similarly define the norm $|\cdot|_\theta$ on $\Sigma_A$ where in the definition of var$_ng$ we use $|i|<n$ instead of $0 \le i <n$. Our main result is the following.

 \begin{thm}\label{thm:symbolic}
Assume that $A$ is irreducible and aperiodic. Then, there exists a constant $c>0$ such that for any $\sigma$-invariant probability measure $\mu$ on $\Sigma_A^+$ and any Lipschitz function $f$ we have
$$\left|\int f d \mu-\int f  d m \right| \le c |f|_\theta(h_{ m}(\sigma) - h_\mu(\sigma))^{1/2}.$$
Moreover, the same result holds for the two sided subshift of finite type $(\Sigma_A,\sigma)$.
\end{thm}

We will give the proof of one sided case and briefly refer how to deduce the two sided counterpart. As a consequence we obtain the analogous result for hyperbolic maps. Before we state the next theorem we would like to introduce the kind of hyperbolic maps we consider.

We consider two kinds of hyperbolic maps: the expanding maps (repellers) and Axiom A diffeomorphisms. 

Let $M$ be a compact, connected, smooth Riemannian manifold and $T:M \to M$ a $C^1$-map. Let $J$ be a compact invariant set so that $T^{-1}J = J$. We say that the pair $(J,T)$ is \emph{a repeller} if $T:J \to J$ is \emph{expanding}, that is, after smooth modification of the Riemannian metric there exists a constant $\theta_0>1$ such that for any $x\in J$ and $v$ in the tangent space $T_x M$, the derivatives satisfy
$$\|D_x T^k(v)\| \ge \theta_0^k \|v\| \text{ for any } k \in \N,$$
and if $J$ is \emph{maximal}, that is, there exists an open neighborhood $V$ of $J$ such that
$$J=\{x \in V: T^n(x) \in V \text{ for all } n \ge 0\}.$$

For a map $T:M \to M$ we define the set $\Omega=\Omega(T)$ of all \emph{non-wandering} points, that is, $x \in \Omega$ if for any neighborhood $U$ of $x$ we have
$$U \cap \bigcup_{k=1}^\infty T^{-k} U \ne \emptyset. $$
Clearly, $\Omega$ is a closed invariant set containing all the periodic points. We say that the set $\Omega$ is \emph{hyperbolic} if there exists a constant $\theta>1$ such that for each $x \in \Omega$ the tangent space $T_x M$ can be decomposed as a direct sum of two $T$-invariant subspaces $W^s(x)$ and $W^u(x)$ so that (after a smooth change of the Riemannian metric, see \cite{HP}) we have
\begin{align*}
\|D_x T^{-k}(v)\| &\le \theta_0^{-k} \|v\| \text{ for any } v \in W^u(x) \text{ and }k \in \N,\\
\|D_x T^k(v)\| &\le \theta_0^{-k} \|v\| \text{ for any } v \in W^s(x) \text{ and }k \in \N.
\end{align*}
A diffeomorphism $T:M \to M$ is called \emph{Axiom A diffeomorphism} if $\Omega(T)$ is hyperbolic and $\Omega(T)=\overline{\{x: x \text{ is periodic}\}}.$ Examples of Axiom A diffeomorphisms include the Anosov diffeomorphisms when $M$ itself is assumed to be hyperbolic \cite{Anosov}. For various examples of Axiom A diffeomorphisms we refer to \cite{Smale}. 

\begin{thm}\label{thm:main}
Let $(J,T)$ be a mixing repeller or $(\Omega(T),T)$ be a mixing Axiom A diffeomorphism.  Then, there exists a constant $c>0$ such that for any $T$-invariant probability measure $\mu$ and any Lipschitz function $f$ with Lipschitz constant $L$ we have
$$\left|\int f d \mu-\int f dm \right| \le c L (h_m(T) - h_\mu(T))^{1/2},$$
where $m$ is the unique measure of maximal entropy.
\end{thm}

From semicontinuity of entropy it follows that if a sequence of invariant probability measures $(\mu_n)$ satisfies $h_{\mu_n}(T) \to h_m(T)$ as $n \to \infty$ then $(\mu_n)$ converges to $m$ in the weak* topology. Theorem~\ref{thm:main} makes this statement effective by comparing the integrals of Lipschitz functions. We also note that although we focused on the measure of maximal entropy the ides should work for alternative measures.

 The results similar to Theorem~\ref{thm:symbolic} and Theorem~\ref{thm:main} were first realised by F.~Polo in \cite{Polo}. Theorem~\ref{thm:symbolic} can be regarded as a generalisation and improvement of some of his results. In fact, he obtains a weaker exponent $1/3$ instead of $1/2$. We do not know if the exponent $1/2$ is sharp. However, if one considers equilibrium measures $\mu$ for certain Lipschitz functions $f$ such that $\mu \ne m$ then clearly $\int f d\mu - \int f d m \ge h_m(\sigma) - h_\mu(\sigma)$ which suggests that the exponent should not be greater than 1.  We also refer to R.~R\"{u}hr \cite{Ruh13} where a similar theorem was obtained for a diagonal action on $\Gamma \backslash \SL(d,\Q_p)$. Our method uses some ideas from the proofs of the related theorems on tori due to Polo \cite{Polo} and improves some of the estimates. To prove \cite[Theorem~4.1.1]{Polo}, Polo makes use of the existence of fine partitions with thin boundary subordinate to the stable subgroup. Such partitions were extensively studied by Einsiedler and Lindenstrauss  in homogeneous dynamics context, see e.g. \cite{EL10}. In our situation, however, we make use of the existence of Markov partitions instead. This way we obtain Theqorem~\ref{thm:main} as a result of Theorem~\ref{thm:symbolic} considering the symbolic representation of the system by a subshift of finite type. 

We now state an application of our main result to estimating Hausdorff dimension of certain exceptional sets.

 For a repeller $(J,T)$ and for any $x_0 \in J$ and $\delta>0$ let
$$E(x_0,\delta):=\{x \in J : T^nx \not \in B(x_0,\delta), \forall n \in \N  \},$$
where $B(x_0,\delta)=\{x \in J: d(x,x_0) < \delta\}.$ For an Axiom A diffeomorphism $(\Sigma(T),T)$ we similarly define the sets $E(x_0,\delta)\subset \Sigma(T)$ for $x_0 \in \Sigma(T)$. Estimation of Hausdorff dimension of these exceptional sets have been studied by many. Using the perturbation theory, asymptotically exact formulas for the Hausdorff dimension $\dim_H E(x_0,\delta)$ of exceptional sets were computed by Ferguson and Pollicott \cite{FP12} in the case of \emph{conformal} repellers, see also \cite{KL09}. We also refer to a similar result \cite{H92} in the context of continued fractions.  On the other hand, for general expanding maps there are several works related to estimating a lower bound for $\dim_H E(x_0,\delta)$ to show that the set of nondense orbits has full Hausdorff dimension, see e.g. \cite{U, AN06,T}. In \cite{U}, Urbanski also considers the lower bound estimates in the settings of Anosov diffeomorphisms and Anosov flows.  In the setting of general expanding maps they consider, Abercrombie and Nair \cite{AN06} point out the difficulty of estimating the upper bound for $\dim_H E(x_0,\delta)$ in terms of $\delta$.

As an application of our main result we estimate an upper bound for the Hausdorff dimension of $E(x_0,\delta)$ for both situations we consider in this article, namely, transitive repellers (not necessarily conformal) and transitive Axiom A diffeomorphisms. As in Theorem~\ref{thm:main} we let $m$ be the measure in $(J,T)$ or $(\Omega(T),T)$ with maximal entropy.

\begin{thm}\label{thm:dim} Let $(J,T)$ be a mixing repeller or $(\Omega(T),T)$ be a mixing Axiom A diffeomorphism. Then, there exists a constant $c>0$ such that for any $x_0 \in J$ and $\delta \ge 0$ we have
$$\dim_H  E(x_0,\delta) \le \dim M - c \delta^2 m(B(x_0,\delta))^{2}.$$
\end{thm}

We will make no claim on the sharpness of the exponent $2$. In the next section we introduce the Parry measure and state the Pinsker inequality. In \S~\ref{sec:fn} we first introduce a sequence of functions $(f_n)$ using a transfer operator. Then, using the Pinsker inequality and properties of $f_n$'s we give the proof of Theorem~\ref{thm:symbolic}. The proof of Theorem~\ref{thm:main} is obtained in \S~\ref{sec:markov} using the Markov partitions and finally the upper estimate for the Hausdorff dimension of the exceptional sets were obtained in the last section.

\subsection*{Acknowledgement}
The author acknowledges the partial supported by EPSRC. He is grateful to Manfred Einsiedler for referring to \cite{Polo}. He also would like to thank Ren\'{e} R\"{u}hr and anonymous referee for useful comments for the preliminary version of the article.

\section{Parry measure and Pinsker inequality}\label{sec:parry}

In this section define the Parry measure $ m$ on $\Sigma_A$, \cite{Parry}. Let $\lambda> 1$ be the largest eigenvalue of $A$. Since $A$ is irreducible, using Perron-Frobenius theory cf. \cite[\S~0.9]{Walter}, we may pick strictly positive left and right eigenvectors $(u_0,u_1, \dots,u_{s-1})$ and $(v_0,v_1,\dots,v_{s-1})$ respectively with $\sum_{i=0}^{s-1} u_i v_i=1$. We set $p_i=u_i v_i$ and $p_{ij}=a_{ij} v_j /\lambda v_i$. Then the Markov measure $  m$ given by the probability vector ${\bf p}=(p_0,p_1,\dots,p_{s-1})$ and the stochastic matrix $(p_{ij})$ is called Parry measure.  

Let $\xi=\{C_i: i \in \Lambda\}$ be the standard partition for $\Sigma_A^+$ where $C_i=\{x \in \Sigma_A^+: x_0=i\}$. We say that a finite or an infinite word $(i_0,i_1,\dots)$ is \emph{admissible} if for any $n$, $A(i_n, i_{n+1})=1.$ For any admissible $(i_0,i_1,\dots,i_k)$ we define the $(k+1)$-cylinder set
$$C(i_0,i_1,\dots,i_k):=\{  x \in \Sigma_A^+ : x_0=i_0,\dots,x_k=i_k\}=\bigcap_{n=0}^k \sigma^{-n} C_{i_n}.$$
We recall that the Parry measure satisfies for any cylinder set
\begin{equation*}
   m (C(i_\ell,i_1,\dots,i_{\ell+k}))=p_{i_\ell} p_{i_\ell i_{\ell+1}} p_{i_{\ell+1} i_{\ell+2}} \cdots p_{i_{\ell+k-1} i_{\ell+k}}.
 \end{equation*}
This gives
\begin{equation}\label{eqn:volC}
  m(C(i_\ell,\dots,i_{\ell+k}))=u_{i_\ell} v_{i_\ell} \prod_{j=0}^{k-1} \frac{a_{i_{\ell+j} i_{\ell+j+1}}v_{i_{\ell+j+1}}}{\lambda v_{i_{\ell+j}}}=\frac{u_{i_\ell} v_{i_{\ell+k}}}{\lambda^k}.
\end{equation}
By setting $a=\min_{i,j} u_i v_j $ and $b=\max_{i,j} u_i v_j$ we obtain
\begin{equation}
a \lambda^{-k}\le   m(C(i_\ell,i_1,\dots,i_{\ell+k})) \le b \lambda^{-k}.
\end{equation}

For any partition $\zeta$ of $\Sigma_A^+$, let $[  x]_{\zeta}:=\bigcap_{  x \in B \in \zeta} B$ denote the atom of $\zeta$ containing $  x$ and $  m_{  x}^{\zeta}$ denote the conditional measure with respect to $\zeta$ supported on $[  x]_\zeta.$ We let $\A=\bigvee_{i=0}^\infty \sigma^{-i} \xi$ and consider the information function 
$$\iota(  x):=I_{  m}(\A | \sigma^{-1}\A)(  x)=-\log   m_{  x}^{\sigma^{-1}\A} ([x]_\A).$$
For more information on conditional measures we refer to \cite[\S~5]{EW}. In general, conditional measures are defined a.e. Therefore, in general $\iota(  x)$ is a measurable function defined only $  m$-a.e. However, in our situation we also would like to integrate the function $\iota$ w.r.t other measures. So, we want to have an everywhere defined measurable function $\iota$ which requires us to define the conditional measures for all $  x \in  \Sigma_A^+$ as we do now. For any $  x \in \Sigma_A^+$, using \eqref{eqn:volC},  we have 
\begin{align*}
  m_{  x}^{\sigma^{-1}\A} ([  x]_\A)&= \lim_{n \to \infty} \frac{   m([  x]_{\bigvee_{i=0}^{n} \sigma^{-i}\xi})}{   m ([  x]_{\bigvee_{i=1}^{n} \sigma^{-i}\xi})}=\lim_{n \to \infty}\frac{  m(C(x_0,x_1,\dots,x_n))}{  m(C(x_1,x_2,\dots,x_n))}\\
&=\frac{u_{x_0} v_{x_n}/\lambda^n}{u_{x_1} v_{x_n}/\lambda^{n-1}}=\frac{u_{x_0}}{\lambda u_{x_1}} .
\end{align*} 
Thus, we have an \emph{everywhere} defined function
$$\iota(  x)=\log \lambda+g(\sigma   x)-g(  x)$$
where $g(  y)=\log u_{ y_0}.$ Hence, we conclude

\begin{lem}\label{lem:iota}
For any invariant probability measure $\mu$ on $\Sigma_A^+$, we have
$$\int \iota d\mu=h_{  m}(\sigma)=\log \lambda.$$
\end{lem}

Let $\Delta_n$ be the $n$-dimensional simplex of probability vectors $  q=(q_1,q_2,\dots,q_n)$ satisfying $q_i \ge 0$ and $\sum_{i=1}^n q_i=1.$ For a fixed vector $  p \in \Delta_n$ with strictly positive entries we define the function
$$\phi: \Delta_n \to \R \text{ by } \phi(  q)=-\sum_{i=1}^n q_i \log \frac{p_i}{q_i},$$
with the convention $0 \log \frac{p_i}{0} = 0.$ Fix the norm $\|  q\|=\sum_i |q_i|$ on $\R^n$. We have

 \begin{lem}[Pinsker Inequality]\label{lem:Delta}
$\phi$ is nonnegative and has a unique $0$ at $  p$. Moreover, for any $  q\in \Delta_n$ we have $$\|  q-  p\| \le \sqrt{2\phi(  q)}.$$
\end{lem}

For the proof we refer to \cite[Lemma 12.6.1]{CT}.

Let $  p,  q \in \Delta_s$ be given by $p_i=p_i(  x)=  m_{  x}^{\sigma^{-1}\A}(C_i)$ and $q_i=q_i(  x)=\mu_{  x}^{\sigma^{-1}\A}(C_i)$. Then,
\begin{multline*}
\int (I_m(\A | \sigma^{-1}\A)(  y)-I_\mu(\A | \sigma^{-1}\A)(  y)) d \mu_{  x}^{\sigma^{-1}\A}(  y)\\
=- \sum_{i=1}^s \mu_{  x}^{\sigma^{-1}\A}(C_i) \log \frac{  m_{  x}^{\sigma^{-1}\A}(C_i)}{\mu_{  x}^{\sigma^{-1}\A}(C_i)}=\phi(q(  x))=:\phi_{  x}(q(  x)).
\end{multline*}
We recall that the function $\phi$ is defined for strictly positive $  p$. In the present situation, we may have $p_i=  m_{  x}^{\sigma^{-1}\A}(C_i)=0$ for some $i$. However, in this situation we also have $q_i=0$. So, we may safely drop the $i$-th term in the definition of $\phi.$ 

The fact that $\int \int I_\mu(\A | \sigma^{-1}\A) d \mu_{  x}^{\sigma^{-1}\A} d\mu(x)=h_\mu(\sigma)$ together with Lemma~\ref{lem:iota} give the following.
\begin{lem}\label{lem:intphi} For any invariant probability measure $\mu$ on $\Sigma_A^+$, we have
$$\int \phi_x(q(x)) d\mu(x)=h_m(\sigma)-h_\mu(\sigma).$$
\end{lem}

\section{A transfer operator and a sequence of Lipschitz functions}\label{sec:fn}

In this section we define the sequence $(f_n)_{n \ge 0}$ of Lipschitz functions using the transfer operator and obtain its useful properties. 

We define the \emph{transfer operator} $\mathcal L:L^1(\Sigma_A^+,\A,m) \to L^1(\Sigma_A^+,\A,m)$ by 
$$\mathcal L f =\frac{d m_f \circ \sigma^{-1}}{d m} \text{ where } dm_f=f dm.$$
It is easy to see that
\begin{equation}\label{eqn:Lexp}
(\mathcal L f) \circ \sigma = E_{  m} (f | \sigma^{-1} \A).
\end{equation}
Define the sequence of functions $f_{n}$ by
$$f_{n}:=\mathcal L^n f=\mathcal L f_{n-1}.$$
We would like to estimate the supremum norms of $f_n$ when $\int f dm=0$. We record (without proof) the following classical result which follows from Ruelle-Perron-Frobenius theorem (see e.g. \cite[Lemma~1.10 ]{B2} and \cite[Theorem 2.2]{PP90} ).

\begin{lem}\label{lem:fnplus1}
There exists a constant $C>0$ and $\rho \in (0,1)$ such that for any Lipschitz function $g$ on $\Sigma_A^+$ with $\int g dm=0$ we have
$$|\mathcal L^n g|_\infty \le  C \rho^n |g|_\theta, \text{ for any } n \ge 0.$$
\end{lem}

We note that Ruelle-Perron-Frobenius theorem gives the estimate $\|\mathcal L^n g\|_\theta \le  C \rho^n \|g\|_\theta$ where $\|g\|_\theta=|g|_\infty+|g|_\theta$. However, since $\int g dm=0$ we can in fact deduce $|g|_\infty= |g-\int g dm|_\infty \le \sup_x \int |g(x)-g(y)| dm(y) \le |g|_\theta$ so that $\|g\|_\theta \le 2|g|_\theta.$

Next, we would like to estimate $|\int f_{n+1}\,d\mu-\int f_n \,d\mu|$. 

\begin{lem} \label{lem:nplus1nestim} 
For any probability invariant measure $\mu$ on $\Sigma_A^+$ and $n \in \N$ we have
$$\left|\int f_{n+1}\,d\mu -\int f_n\, d\mu\right|\le \sqrt{2} |f_n|_\infty(h_{  m}(\sigma)-h_\mu(\sigma))^{1/2}.$$
\end{lem}

\begin{proof}
From the defining property of conditional measures we have that the conditional expectation for an integrable function $g$ satisfies 
 $$E_{  m} (g | \sigma^{-1} \A)(  x)=\int_{\Sigma_A^+} g(  y) \,d{  m}_{  x}^{\sigma^{-1}\A}(  y).$$
For any $j \in \Lambda$ let $S_j=\{i \in \Lambda : A(i,j)=1\}$. Then,
\begin{align*}
E_\mu(f_n | \sigma^{-1}\A)(  x)&=\sum_{i \in S_{x_1}} f_n(  x^{(i)}) \mu_{  x}^{\sigma^{-1}\A}(C_i),\\
E_m(f_n | \sigma^{-1}\A)(  x)&=\sum_{i \in S_{x_1}} f_n(  x^{(i)}) m_{  x}^{\sigma^{-1}\A}(C_i).
\end{align*}
Combining with Lemma~\ref{lem:Delta} we get that
\begin{align*}
|E_{  m}(f_n | \sigma^{-1}\A)(  x)-E_\mu(&f_n | \sigma^{-1}\A)(  x)|\\
&\le \sum_{i \in S_{x_1}} |f_n(  x^{(i)})| |  m_{  x}^{\sigma^{-1}\A}(C_i)-\mu_{  x}^{\sigma^{-1}\A}(C_i)|\\
&\le |f_n|_\infty\sum_{i \in \Lambda} |  m_{  x}^{\sigma^{-1}\A}(C_i)-\mu_{  x}^{\sigma^{-1}\A}(C_i)|\\
&=|f_n|_\infty \|p(  x)-q(  x)\| \le  |f_n|_\infty \sqrt{2\phi_{  x}(q(  x))} .
\end{align*}
From \eqref{eqn:Lexp} and $\sigma$-invariance of $\mu$ we see that $\int f_{n+1} d\mu=\int E_m(f_n|\sigma^{-1}\A) d\mu$. So,
\begin{multline*}
\left|\int f_{n+1}d\mu -f_n d\mu\right|=\left|\int E_{  m}(f_n|\sigma^{-1}\A)- E_\mu(f_n|\sigma^{-1}\A) d\mu\right|\\
\le \sqrt{2}|f_n|_\infty \int \sqrt{\phi_{x}(q(  x))}d\mu(  x).
\end{multline*}
Finally, we note that the Cauchy-Schwarz inequality yields
$$ \int \sqrt{\phi_{  x}(q(  x))}d\mu(  x) \le \left( \int \phi_x(q(x)) d\mu(x)\right)^{1/2}.$$
Hence, together with Lemma~\ref{lem:intphi} we arrive at
$$\left|\int f_{n+1}d\mu -f_n d\mu\right| \le \sqrt{2}|f_n|_\infty (h_m(\sigma)-h_\mu(\sigma))^{1/2}.\qedhere$$
\end{proof}

We are now ready to prove Theorem~\ref{thm:symbolic} using the above two lemmas.
 
\begin{proof}[Proof of Theorem~\ref{thm:symbolic}]
It suffices to prove Theorem~\ref{thm:symbolic} for Lipschitz functions $f$ with $\int f dm=0.$ As before we set $f_n=\mathcal L^n f$ for $n \ge 0$. From Lemma~\ref{lem:fnplus1} we see that $\int f_n \,d\mu$ converges to $0=\int f \,d  m $ which gives
\begin{align*}
\left| \int f d\mu - \int f d  m\right|=\lim_{n \to \infty} \left|\int f d\mu - \int f_n d\mu\right| \le \sum_{n=0}^\infty \left| \int f_{n+1} d\mu - \int f_n d\mu\right|.
\end{align*}
Now, using the estimate from Lemma~\ref{lem:nplus1nestim} together with Lemma~\ref{lem:fnplus1} we conculde
\begin{align*}
\left| \int f d\mu - \int f d  m\right| &\le \sum_{n=0}^\infty \sqrt{2} |f_n|_\infty(h_{  m}(\sigma)-h_\mu(\sigma))^{1/2}\\
&\le \sum_{n=0}^\infty \sqrt{2} C \rho^n |f|_\theta (h_{  m}(\sigma)-h_\mu(\sigma))^{1/2}\\
&= \frac{\sqrt{2} C}{1-\rho}  |f|_\theta (h_{  m}(\sigma)-h_\mu(\sigma))^{1/2}.
\end{align*}
This proves the theorem for $(\Sigma_A^+, \sigma)$. The proof for the two sided subshift of finite type $(\Sigma_A,\sigma)$ follows analogously.  In this case, one needs to replace Lemma~\ref{lem:fnplus1} with its two sided counterpart which follows from the current Lemma~\ref{lem:fnplus1} using an approximation argument ($\rho$ gets replaced by possibly larger one). For more details see e.g. proof of Proposition~2.4 in \cite{PP90}. 

\end{proof}

 \section{Markov partitions for hyperbolic maps}\label{sec:markov}

In this section we first introduce Markov partitions associated to the hyperbolic maps considered in the introduction. The Markov partitions give us the related symbolic representation of our dynamical system.

It is well known that repellers $(J,T)$ and Axiom A diffeomorphisms $(\Omega,T)$ admit Markov partitions, see e.g. \cite{S1,S2,B1,B2}. We now recall some basics about Markov partitions associated to these systems. 

A repeller $(J,T)$ is conjugated to a one-sided subshift of finite type while the Axiom A diffeomorphism $(\Omega, T)$ is conjugated to a two-sided subshift of finite type. First, let us consider the case when we have a repeller $(J,T)$. Let $\mathcal R = \{R_1,\dots,R_s\}$ be a Markov partition for $(J,T)$ where $R_i$'s are closed sets with mutually disjoint interiors. As before we set $\Lambda=\{1,2,\dots,s\}$. We have
$$m(R_i \backslash R_i^\circ) = 0, i\in\Lambda \text{ and } T(R_i)=\bigcup_{j \in \Lambda_i} R_j,$$
for some suitable subset $\Lambda_i$ of $\Lambda$. By refining the Markov partition, if necessary, we may assume that $T$ is invertible on each partition element $R_i$.  Let $A$ be the $s \times s$ transition matrix associated to $\mathcal R$. This means that all entries $A(i,j)$ are in $\{0,1\}$, and $A(i,j)=1$ if and only if $T(R_i) \cap R_j^\circ \neq \emptyset$ if and only if $j \in \Lambda_i.$ For any admissible $(i_0,i_1,\dots,i_k)$ we define the $(k+1)$-cylinder set in $J$ by
$$R(i_0,i_1,\dots,i_k):=\bigcap_{j=0}^k T^{-j} R_{i_j},$$
which is nonempty. We define a map $\pi: \Sigma_A \to J$ via
$$\pi(  x)=\bigcap_{k=0}^\infty R(x_0,x_1,\dots,x_k).$$ 
This defines a factor map from $(\Sigma_A^+,\sigma)$ to $(J,T)$ as $T\circ \pi = \pi \circ \sigma.$

In the case when $(\Sigma, T)$ is an Axiom A diffeomorphism, by abuse of notation, we again let $\mathcal R = \{R_1,\dots,R_s\}$ be the associated Markov partition. Again, let $A$ be the associated transition matrix. For any admissible $(i_{-\ell},\dots,i_k)$  we define
$$R(i_{-\ell},\dots,i_k):=\bigcap_{j=-\ell}^k T^{j} R_{i_j},$$
and consider the map $\pi: \Sigma_A \to \Omega(T)$ via
$$\pi(  x)=\bigcap_{k=0}^\infty \overline{R(x_{-k},\dots,x_k)}.$$ 
 As a result, we get a factor map $\pi: \Sigma_A \to \Omega(T)$ so that $T\circ \pi=\pi \circ \sigma.$

 Moreover, the push-forward of the measure $  m$ under $\pi$ gives $\pi_*   m = m$ where as before $m$ is the unique measure of maximal entropy in the settings we consider.

We end this section by giving the proof of Theorem~\ref{thm:main}. We need the following classical lemma.

\begin{lem}\label{lem:measure}
Any $T$-invariant probability measure $\mu$ on $J$ can be lifted to a $\sigma$-invariant probability measure $\tilde\mu$ on $\Sigma_A^+$ so that $\pi_* \tilde \mu =\mu.$ Similarly, any $T$-invariant probability measure $\mu$ on $\Sigma$ can be lifted to a $\sigma$-invariant probability measure $\tilde\mu$ on $\Sigma_A$ so that $\pi_* \tilde \mu =\mu.$
\end{lem}

\begin{proof}[Proof of Theorem~\ref{thm:main}]
We will only prove the theorem for $(J,T)$ as the case of $(\Sigma,T)$ is analogous. Let $f$ be a Lipschitz function on $J$ with $L(f)=L$ and let $\mu$ be any $T$-invariant probability measure on $J$. Recall the constant $\theta_0>1$ appeared in the definition of repeller. It is easy to see that if we fix $\theta \in (\theta_0^{-1},1)$, then the function $f\circ\pi$ on $\Sigma_A^+$ is a Lipschitz function with Lipschitz constant $L$ w.r.t. metric $d_\theta$. We use Lemma~\ref{lem:measure} to lift the measure $\mu$ to a $\sigma$-invariant measure $ \tilde \mu$ on $\Sigma_A^+$. Applying Theorem~\ref{thm:symbolic} we get
$$\left |\int  f\circ\pi d\tilde\mu-\int  f\circ \pi d  m \right| \le L (h_{  m}(\sigma)-h_{  \mu}(\sigma))^{1/2}.$$
Note that $\log \lambda=h_{  m}(\sigma)=h_{m}(T)$, $h_{\tilde \mu}(\sigma) \ge h_{\mu}(T)$, and 
$$\int_{\Sigma_A} f\circ\pi d \tilde\mu =\int_J f d\pi_* \mu=\int_J f d\mu.$$ Similarly, $\int_{\Sigma_A}  f\circ \pi  d  m=\int_J f dm$. Therefore,
\begin{multline*}
\left |\int_{J}  f d\mu-\int_{J}  f dm \right|=\left |\int_{\Sigma_A}  f\circ\pi d\tilde \mu-\int_{\Sigma_A}  f\circ \pi d  m \right| \\
\le  L (h_{  m}(\sigma)-h_{ \tilde \mu}(\sigma))^{1/2}\le    L (h_{m}(T)-h_{ \mu}(T))^{1/2}.
\end{multline*}
\end{proof}

\section{Estimate for the Hausdorff dimension}
In this section we use Theorem~\ref{thm:main} to prove Theorem~\ref{thm:dim}. Implicitly, we relate dimension estimates to topological entropy. We refer to a very recent work \cite{FS14} where this relation was made precise for hyperbolic automorphisms of two dimensional torus. 

We treat both cases together. Clearly the set $ E(x_0,\delta)$ is compact and $T$-invariant. We first estimate the topological entropy $h(T |_{E(x_0,\delta)})$. From Theorem~\ref{thm:main} we have that if an invariant measure $\mu$ satisfies $\mu(B(x_0,\delta))=0$ and if $f \ge 0$ is a Lipschitz function with support in $B(x_0,\delta)$ then
$$h_\mu(T) \le h_m(T) - \left(\frac{\int f dm}{cL}\right)^2.$$
To optimise the upper bound we need to choose a suitable function $f$. For example, we may choose a function $f$ which describes the cone with height $\delta L$ and radius in the base equal to $\delta$. Then, clearly $f$ is a Lipschitz function with Lipschitz constant $L$ and $\int f dm \asymp m(B(x_0,\delta)) \delta L$ so that for large values of $L$
$$\left(\frac{\int f dm}{cL}\right)^2 \gg \delta^2 m(B(x_0,\delta))^2.$$
Thus, from the Variational Principle we conclude that for some constant $c>0$ independent of $\delta>0$ and $x_0\in M$  
\begin{equation}
h(T |_{E(x_0,\delta)}) \le h_m(T) - c \delta^2 m(B(x_0,\delta))^2= \log \lambda- c \delta^2m(B(x_0,\delta))^2 .
\end{equation}
On the other hand, for any $h>h(T|_{E(x_0,\delta)})$ and $\epsilon>0$ there exists $N_0$ such that for any $k>N_0$, the maximum cardinality of an $(k,\epsilon)$-separated set in $E(x_0,\delta)$ is less than $e^{h k}$. In particular, if we let
$$\mathcal R_k:=\{R(i_1,\dots,i_k) : R(i_1,\dots,i_k) \cap  E(x_0,\delta)\neq \emptyset\}$$
then, $|\mathcal R_k| \le e^{kh}.$ 

From continuity of $T$ and compactness of $X$, where $X$ is either $J$ or $M$, there exists $\Theta \ge \theta_0$ such that $d(Tx,Ty)\le \Theta d(x,y).$ On the other hand, we recall that the Parry measure satisfies
 $$a \lambda^{-n}\le m(R(i_0,i_1,\dots,i_n)) \le b \lambda^{-n}.$$
It follows that each cylinder set $R(i_0,i_1,\dots,i_n)$ can be covered with at most $a \lambda^{-n}/\Theta^{-n\dim M}$ balls of radius $\Theta^{-n}$, for some possibly modified constant $a$ (independent of $n \in \N$). Hence,
\begin{lem}
For any $h>h(T|_{E(x_0,\delta)})$ and sufficiently large $n$, the set $E(x_0,\delta)$ can be covered by $a e^{nh -n\log \lambda +n \dim M \log \Theta}$ balls of radius $\Theta^{-n}.$
\end{lem}
\begin{proof}
Continuing the above arguments, we get that the number of covers needed is
$$ae^{nh} \lambda^{-n}/\Theta^{-n\dim M}=a e^{nh -n\log \lambda +n \dim M \log \Theta}.$$
\end{proof}
Thus, we get
\begin{align*}
\dim_H E(x_0,\delta)\le   \dim E(x_0,\delta) &\le \lim_{n \to \infty}\frac{\log (a e^{nh -n\log \lambda +n \dim M \log \Theta}) }{-\log \Theta^{-n}}\\
&=\lim_{n \to \infty}\frac{nh -n\log \lambda +n \dim M \log \Theta }{n\log\Theta}\\
&=\dim M- \frac{\log \lambda - h}{\log \Theta}.
\end{align*}
This is true for any $h>h(T|_{E(x_0,\delta)})$ where $h(T|_{E(x_0,\delta)}) \le \log \lambda - c \delta^2m(B(x_0,\delta))^{2}.$ Thus, replacing $c/\log \Theta$ by $c$, we conclude that
$$ \dim_H E(x_0,\delta)\le \dim M - c\delta^2m(B(x_0,\delta))^2.$$
This concludes the proof of Theorem~\ref{thm:dim}.

\end{document}